\documentclass[15pt]{article}

\usepackage{amsmath,amsfonts,amssymb,amsthm,amscd}
\numberwithin{equation}{section}

\newcounter{stepctr}
{\end{list}}

\newtheorem{thm}{Theorem}[section]
\newtheorem{prop}[thm]{Proposition}
\newtheorem{cor}[thm]{Corollary}

\theoremstyle{definition}
\newtheorem{dfn}[thm]{Definition}
\newtheorem{ex}[thm]{Example}

\newtheorem{rema}[thm]{Remark}
\newtheorem{prob*}{ Problem}
 \newtheorem{lem}[thm]{Lemma}

\newcommand{\demo}{\begin{proof}}

\newcommand{\N}{\mathbb{N}}
\newcommand{\NN}{\mathcal{N}}

\newcommand{\C}{\mathbb{C}}

\def\fx{{\mathcal F}(X) }

\def\ll^2{{\mathcal L}(\ell^2(\N)}

\def\rx{{\mathcal R}(X) }
\def\f^0x{{\mathcal F^0}(X) }
\def\qx{{\mathcal Q}(X) }
\def\kx{{\mathcal K}(X) }
\def\nx{{\mathcal N}(X) }

\title
{\bf  On   the  property  $(Z_{E_{a}})$    }
\author{    H. Zariouh }
\date{ }
\begin{document}

\maketitle \thispagestyle{empty}

\begin{abstract}\noindent\baselineskip=15pt
The paper introduces the notion of properties $(Z_{\Pi_{a}})$ and  $(Z_{E_{a}})$  as   variants of Weyl's theorem and Browder's theorem for
bounded linear operators acting  on  infinite dimensional Banach spaces. A characterization of these properties  in
terms of localized  single valued  extension property is
 given, and
 the perturbation   by commuting Riesz operators is also studied. Classes of operators are considered
as illustrating examples.

\end{abstract}

\baselineskip=15pt
 \footnotetext{\small \noindent  2010 AMS subject
classification: Primary 47A53, 47A55, 47A10, 47A11. \\
\noindent Keywords :  Property
$(Z_{E_{a}}),$ SVEP, Weyl's theorem, Riesz operator.}\baselineskip=15pt

\section{Introduction}

In 1909 H.Weyl \cite{W} examined the spectra of all compact perturbation of
 a self-adjoint operator on a Hilbert space and found that their intersection consisted precisely of
  those points of the spectrum which were not isolated eigenvalues of finite multiplicity. Today this
   classical result may be stated by saying that the spectral points of a self- adjoint operator which
   do not belong to Weyl spectrum are precisely the eigenvalues of finite multiplicity which are isolated points
    of the spectrum. This Weyl's theorem has been extended from self- adjoint operators to several other classes of
    operators and many new variants
 have been obtained by many researchers (\cite{BE-Kachad}, \cite{BK}, \cite{BZ3}, \cite{GK},  \cite{KK},  \cite{Ra}).

 This paper is a continuation of our recent investigations in the subject
of Weyl type theorems. We introduce and study the new variants of Weyl's theorem and Browder's theorem. The essential results
obtained are summarized in the diagram presented in the  end of the second section of this
paper. For further definitions and symbols we also refer the reader
to \cite{BK}, \cite{BZ3} and  \cite{ZZ}.

We begin with some standard notations of Fredholm theory. Throughout this paper let $\mathcal{B}(X)$ denote the algebra of
all bounded linear operators on an infinite-dimensional complex Banach space $X.$ For an operator $T\in \mathcal{B}(X),$
we denote by $T^*,$  $\sigma(T),$ $\NN(T)$ and $\mathcal{R}(T)$ the dual of $T,$ the spectrum of $T,$
 the null space of $T$ and  the  range space of $T,$ respectively. If $\dim\NN(T)<\infty$ and  $\dim\NN(T^*)<\infty$, then $T$ is
 called a \emph{Fredholm} operator and its index is defined by $\mbox{ind}(T)=\dim\NN(T)-\dim\NN(T^*).$
A \emph{Weyl} operator is a Fredholm operator of index $0$ and the Weyl spectrum is defined by  $\sigma_{W}(T)=\{\lambda\in\mathbb{C}:T-\lambda I \mbox{
is not a Weyl operator}\}.$

For a bounded linear operator $T$ and $n\in\N,$
 let $T_{[n]}: \mathcal{R}(T^n)\rightarrow \mathcal{R}(T^n)$  be the restriction of $T$ to  $\mathcal{R}(T^n).$
  $T \in L(X)$ is said to be \emph{B-Weyl} if for some integer $n \geq 0$ the range $\mathcal{R}(T^n)$
is closed and $T_{[n]}$ is Weyl;  its index is defined as the index of the Weyl  operator $T_{[n]}.$
The respective \emph{B-Weyl spectrum} is defined by $\sigma_{BW}(T)=\{\lambda\in\mathbb{C}:T-\lambda I \mbox{
is not a B-Weyl operator}\},$ see \cite{BE1}.

The  \textit{ascent} $a(T)$ of  an operator $T$ is defined by
 $a(T)=\mbox{inf} \{ n\in \mathbb{N}: \NN(T^n)=\NN(T^{n+1})\}$,
and the  \textit{descent}   $ \delta(T)$ of $T$ is defined by
 $\delta(T)= \mbox{inf} \{ n \in \mathbb{N}: \mathcal{R}(T^n)= \mathcal{R}(T^{n+1})\},$ with $ \mbox{inf}\, \emptyset= \infty.$ 
 An operator $T \in\mathcal{B}(X)$ is called Browder if it is Fredholm of finite ascent, and
finite descent and the respective Browder spectrum is defined by $\sigma_{b}(T)=\{\lambda \in\mathbb{C}:T-\lambda I \mbox{
is not a Browder operator}\}.$
According to \cite {H}, a complex number $\lambda\in\sigma(T)$ is a
\textit{pole} of the
 resolvent of $T$  if $T-\lambda I$ has  finite ascent and finite
 descent, and in  this case they are equal. We recall \cite{BK} that a  complex number $\lambda\in\sigma_a(T)$
 is a \textit{left pole} of $T$ if $a(T-\lambda I)<\infty$ and $\mathcal{R}(T^{a(T-\lambda I)+1})$ is
 closed.
 In addition, we have
the following usual notations that will be needed later:\\
\noindent  {\bf\emph{{Notations and symbols:}}}
\smallskip

\noindent $\fx$: the ideal of finite rank  operators in $\mathcal{B}(X),$ \\
 $\kx$: the ideal of compact operators in $\mathcal{B}(X),$ \\
 $\nx$: the class of nilpotent operators on $X,$ \\
 $\qx$: the class of quasi-nilpotent operators on $X,$ \\
$\rx$: the class of Riesz  operators acting on  $X,$ \\
\noindent $\mbox{iso}\,A$: isolated points of a  subset $A\subset \mathbb{C},$\\
 \noindent $\mbox{acc}\,A$: accumulations  points of a subset $A\subset \mathbb{C},$\\
\noindent $D(0, 1)$: the closed unit disc in $\mathbb{C},$\\
$C(0, 1)$: the  unit circle of $\mathbb{C},$\\
$\Pi(T)$: poles of $T,$\\
$\Pi^0(T)$: poles of $T$ of finite rank,\\
$\Pi_a(T)$: left  poles of $T,$\\
$\sigma_{p}(T)$:  eigenvalues of $T,$\\
$\sigma_{p}^0(T)$: eigenvalues of $T$ of finite multiplicity,\\
$E^0(T):=\mbox{iso}\,\sigma(T)\cap\sigma_{p}^0(T),$ \\
$E(T):=\mbox{iso}\,\sigma(T)\cap\sigma_{p}(T),$\\
$E_a(T):=\mbox{iso}\,\sigma_a(T)\cap\sigma_{p}(T),$\\
$\sigma_{b}(T)=\sigma(T)\setminus\Pi^0(T)$:  Browder spectrum of $T,$\\
$\sigma_{W}(T)$:  Weyl spectrum of $T,$\\
$\sigma_{BW}(T)$:  B-Weyl spectrum of $T,$\\
the symbol $\bigsqcup$ stands for the disjoint union.

\begin{dfn} \label{dfn0}  \cite{BK}, \cite{harte}, \cite{W} Let $T\in \mathcal{B}(X).$ $T$ is said to satisfy\\
i) Weyl's theorem if
$\sigma(T)\setminus\sigma_{W}(T)=E^0(T);$ ($W$ for brevity).\\
ii) Browder's theorem if  $\sigma(T)\setminus\sigma_{W}(T)=\Pi^0(T);$ ($B$ for brevity).\\
iii) generalized Weyl's theorem if  $\sigma(T)\setminus\sigma_{BW}(T)=E(T);$ ($gW$ for brevity).\\
iv) generalized Browder's theorem if  $\sigma(T)\setminus\sigma_{BW}(T)=\Pi(T);$ ($gB$ for brevity).
\end{dfn}

\begin{dfn}\label{dfn1}\cite{BZ3},\cite{ZZ} Let $T\in \mathcal{B}(X).$ $T$ is said to satisfy\\
i) Property $(gab)$ if  $\sigma(T)\setminus\sigma_{BW}(T)=\Pi_a(T).$\\
ii) Property $(gaw)$ if  $\sigma(T)\setminus\sigma_{BW}(T)=E_a(T).$\\
iii) Property $(ab)$ if $\sigma(T)\setminus\sigma_{W}(T)=\Pi_a^0(T).$\\
iv)  Property $(aw)$ if  $\sigma(T)\setminus\sigma_{W}(T)=E_a^0(T).$\\
v) Property $(Bab)$ if  $\sigma(T)\setminus\sigma_{BW}(T)=\Pi_a^0(T).$\\
vi) Property $(Baw)$ if  $\sigma(T)\setminus\sigma_{BW}(T)=E_a^0(T).$\end{dfn}

The relationship between properties and  theorems given  in the precedent definitions  is
summarized in  the following diagram. (arrows signify implications and numbers near the arrows are
references to the bibliography therein).

 \vspace{5pt}

\vbox{
\[
\begin{CD}@.(Baw)\\
@.@VV\mbox{{\scriptsize\cite{BZ3}}}V\\(gaw)
@>\mbox{{\scriptsize\cite{ZZ}}}>> (aw) @>\mbox{{\scriptsize\cite
{BZ1}}}>>\mbox{$W$}@< \mbox{{\scriptsize\cite
{BK}}}<<\mbox{$gW$} \\
@VV\mbox{{\scriptsize\cite{ZZ}}}V@VV\mbox{{\scriptsize\cite{BZ3}}}V
@VV\mbox{{\scriptsize\cite{Barnes}}}V@VV\mbox{{\scriptsize\cite{BE1}}}V\\
 (gab)
@>>\mbox{{\scriptsize\cite{BZ3}}}> (ab)
@>>\mbox{{\scriptsize\cite{BZ3}}}> \mbox{$B$}& \Longleftrightarrow_{\mbox{{\scriptsize\cite{AmZg1}}}} &\mbox{  $gB$}\\
@.@AA\mbox{{\scriptsize\cite{ZZ}}}A\\
@.(Bab)\\
\end{CD}
\]
}
Moreover,   counterexamples were given to  show that the reverse of each implication in the diagram is not true. Nonetheless,
it was proved  that under some extra assumptions, these implications
 are equivalences.

\section{ Properties $(Z_{\Pi_{a}})$ and $(Z_{E_{a}})$  }

We define the properties $(Z_{\Pi_{a}})$ and $(Z_{E_{a}})$  as follows:

\begin{dfn}A bounded linear operator $T\in \mathcal{B}(X)$ is said to satisfy property
 $(Z_{E_{a}})$ if $\sigma(T)\setminus\sigma_{W}(T)= E_a(T),$ and
is said to satisfy property
 $(Z_{\Pi_{a}})$ if
 $\sigma(T)\setminus\sigma_{W}(T)= \Pi_a(T).$ \end{dfn}

\begin{ex} Hereafter, we denote by $R$ the unilateral right shift
operator defined on the $\ell^2(\N)$ by $R(x_1, x_2, x_3, \ldots)=(0, x_1, x_2, x_3, \ldots).$
\begin{enumerate}
\item  It is well known that $\sigma(R)=D(0, 1),$
$\sigma_{W}(R)=D(0, 1)$ and
$E_a(R)=\Pi_a(R)=\emptyset.$  So $R$ satisfies the property
$(Z_{E_{a}})$ and the property $(Z_{\Pi_{a}}).$
\item  Let $V$ denote the Volterra operator on the
Banach space $C[0, 1]$ defined by $V(f)(x)=\int_0^x f(t)dt \mbox{for all} f\in C[0, 1].$ $V$ is injective\texttt{} and
quasinilpotent. $\sigma(V)=\sigma_{W}(V)=\{0\}$ and $\Pi_a(V)=E_a(V)=\emptyset.$ So $V$  satisfies the properties $(Z_{E_{a}})$ and   $(Z_{\Pi_{a}}).$
\end{enumerate}
 \end{ex}

\begin{lem}\label{lem0} Let  $T\in \mathcal{B}(X).$ If $T$ satisfies property $(Z_{E_{a}}),$ then
\[E_a(T)=E_a^0(T)=\Pi_a^0(T)=\Pi_a(T)=\Pi^0(T)=\Pi(T)=E^0(T)=E(T).\] \end{lem}

\begin{proof}
Suppose that $T$ satisfies property $(Z_{E_{a}})$,
then $\sigma(T)=\sigma_{W}(T)\sqcup E_a(T)$. Thus  $\mu\in
E_a(T)\Longleftrightarrow\mu\in\mbox{iso}\,\sigma_a(T)\cap\sigma_{W}(T)^C$
$\Longrightarrow \mu\in\Pi_a^0(T),$ where $\sigma_{W}(T)^C$ is the
complement of the Weyl spectrum of $T$. Hence
$E_a(T)=E_a^0(T)=\Pi_a^0(T)=\Pi_a(T)$, $\Pi(T)=\Pi^0(T)$ and
$E(T)=E^0(T).$ Consequently, $\sigma(T)=\sigma_{W}(T)\sqcup E_a^0(T).$ This implies  that
$E^0(T)=\Pi^0(T)$. Hence
$E_a(T)=E_a^0(T)=\Pi_a^0(T)=\Pi_a(T)$ and $\Pi^0(T)=\Pi(T)=E^0(T)=E(T).$ Since the inclusion  $\Pi(T)\subset\Pi_a(T)$ is always
 true, it suffices to  show its opposite. If $\mu\in\Pi_a(T),$ then $a(T-\mu I)$ is finite and since $T$ satisfies
 property $(Z_{E_{a}}),$ it follows that $\mu\in \Pi(T)$ and hence the equality desired.
 \end{proof}

\begin{cor} Let $T\in \mathcal{B}(X).$ The following statements are equivalent:\\
i)  $T$ satisfies property $(Z_{E_{a}});$ \\
ii) $T$ satisfies Weyl's theorem and $E^0(T)=E_a(T);$\\
iii) $T$ satisfies  Browder's theorem and $\Pi^0(T)=E_a(T).$\\
iv) $T$ satisfies generalized Weyl's Theorem and $E^0(T)=E_a(T);$
\end{cor}

\begin{proof} The equivalence between the first three statements is clear.\\
 To prove the equivalence  between (i) and  (iv). If $T$ satisfies property $(Z_{E_{a}}),$ then $T$ satisfies
 Browder's theorem and then generalized Browder's theorem too. Thus from Lemma \ref{lem0}, $T$
 satisfies generalized Weyl's theorem and $E^0(T)=E_a(T).$ Conversely, suppose that  $T$ satisfies generalized Weyl's
  and $E^0(T)=E_a(T).$ From \cite[Theorem 3.9]{BK}, $T$ satisfies Weyl's theorem $\sigma(T)\setminus\sigma_{W}(T)= E^0(T)=E_a(T).$
   So $T$ satisfies property $(Z_{E_{a}}).$

\end{proof}

Following  \cite{KK}, an operator $T\in \mathcal{B}(X)$ is said to satisfy property $(k)$
if $\sigma(T)\setminus\sigma_{W}(T)= E(T).$ For the definition of property $(k),$ see also  the reference \cite{BE-Kachad} in which this property
is named $(W_E).$  From  Lemma \ref{lem0} we have immediately the
next corollary:

\begin{cor} \label{cor2}Let $T\in \mathcal{B}(X)$. The following statements are equivalent:\\
i) $T$ satisfies property $(Z_{E_{a}});$\\
ii) $T$ satisfies property $(Z_{\Pi_{a}})$ and $E_a(T)=\Pi_a(T);$\\
iii) $T$ satisfies property $(k)$ and $E_a(T)=E(T).$\end{cor}

\begin{ex}\label{ex1} Generally, we cannot expect that property $(Z_{E_{a}})$ holds for an
operator satisfying property $(Z_{\Pi_{a}})$ or property $(k)$,
as we can see in  the following example.

\begin{enumerate} \item Let $Q\in \mathcal{B}(X)$ be a quasi-nilpotent operator acting on an  infinite
dimensional Banach space $X$    such that $\mathcal{R}(Q^n)$ is non-closed for
all $n\in \N$ and  let $T=0\oplus Q$ defined on the Banach space $X\oplus X.$
Clearly,   $\sigma_{W}(T)=\sigma_{BW}(T)=\sigma(T)=\{0\}$,
 $E_a(T)=\{0\}$ and
$\Pi_a(T)=\emptyset$. So $T$ satisfies property $(Z_{\Pi_{a}})$, but
it does not satisfy property $(Z_{E_{a}})$.
\item Let $T$ be the operator given by the direct sum of the unilateral
right shift operator  $R$ and the projection operator $U$
defined   in the first point of Remark \ref{rema1} below.
Then $\sigma(T)=D(0, 1)$, $\sigma_{W}(T)=D(0, 1)$, $E(T)=\emptyset$.
So $T$ satisfies property $(k),$ but it  does not satisfy
property $(Z_{E_{a}})$, since $E_a(T)=\{0\}$.\end{enumerate}
\end{ex}

In the following theorem we establish a relationship between property $(Z_{E_{a}})$ and the properties $(gaw),$ $(aw),$  $(Baw)$
 (see Definition\ref{dfn1}).

\begin{thm}\label{thm1} Let $T\in \mathcal{B}(X)$. The following statements are equivalent:\\
i) $T$ satisfies property
$(Z_{E_{a}});$\\
ii)  $T$ satisfies property $(gaw)$ and $\sigma_{BW}(T)=\sigma_{W}(T);$\\
iii) $T$ satisfies property $(aw)$ and $E_a(T)=E_a^0(T);$\\
iv) $T$ satisfies property $(Baw)$ and $E_a(T)=E_a^0(T);$\end{thm}

\begin{proof}(i) $\Longleftrightarrow$ (iii) Suppose that $T$ satisfies property $(Z_{E_{a}})$,
then from Lemma \ref{lem0}, $\sigma(T)=\sigma_{W}(T)\sqcup E_a(T)=\sigma_{W}(T)\sqcup E_a^0(T).$ So $T$ satisfies property $(aw)$ and $E_a(T)=E_a^0(T).$
  The converse is clear.\\
(i) $\Longleftrightarrow$ (ii) If $T$ satisfies property $(Z_{E_{a}})$, then it satisfies property $(aw).$ Since by Lemma \ref{lem0} we have
$E_a(T)=\Pi(T),$ it follows
 from \cite[Theorem 2.2]{BZ1} that $T$ satisfies property $(gaw),$ and this entails  that $\sigma_{BW}(T)=\sigma(T)\setminus
E_a(T)=\sigma_{W}(T).$ The converse is obvious.\\
(i) $\Longleftrightarrow$ (iv) If $T$ satisfies property $(Z_{E_{a}}),$ then $\sigma_{BW}(T)=\sigma_{W}(T)$ and $E_a(T)=E_a^0(T).$
So  $\sigma(T)\setminus\sigma_{BW}(T)= E_a^0(T),$ i.e. $T$ satisfies property $(Baw).$ Conversely,
suppose that $T$ satisfies property  $(Baw)$ and $E_a(T)=E_a^0(T).$ By \cite[Theorem 3.3]{ZZ} we have
 $\sigma_{BW}(T)=\sigma_{W}(T).$ Thus  $E_a(T)=E_a^0(T)=\sigma(T)\setminus\sigma_{BW}(T)=\sigma(T)\setminus\sigma_{W}(T),$
 and  $T$ satisfies property $(Z_{E_{a}}).$\end{proof}

\begin{rema}\label{rema1}  From Theorem \ref{thm1}, if $T\in \mathcal{B}(X)$ satisfies
property  $(Z_{E_{a}})$ then it satisfies property $(\delta);$ where $\delta \in\{ gaw, aw, Baw\}$. However,
the converse  in general is not true. To see this,
\begin{enumerate}
\item Let  $U\in
L(\ell^2(\N)$ be defined by $U(x_1, x_2, x_3, ...)=(0, x_2, x_3,
...)$. Then $\sigma(U)=\{0, 1\},$ $\sigma_{W}(U)=\{1\}$,
$E_a(U)=\{0, 1\}$ and $\sigma_{BW}(U)=\emptyset$. So
 $U$ satisfies property $(gaw)$ and then property $(aw).$ But it
does not satisfy property $(Z_{E_{a}})$, because
$\sigma(U)\setminus\sigma_{W}(U)\neq E_a(U).$ Here $E_a^0(U)=\{0\}.$
\item   On the Banach space    $\ell^2(\N)\oplus \ell^2(\N),$ we consider the operator $T$  defined by $T=0\oplus R.$  We have   $T$
satisfies property $(Baw),$ since $\sigma(T)=\sigma_{BW}(T)=D(0, 1)$ and $E_a^0(T)=\emptyset.$ But it does not satisfy property
  $(Z_{E_{a}}),$ since   $\sigma_{W}(T)=D(0, 1)$ and  $E_a(T)=\{0\}.$
\end{enumerate}
\end{rema}

\begin{lem}\label{lem2} Let $T\in \mathcal{B}(X).$ If $T$  satisfies property $(Z_{\Pi_{a}}),$ then
\[\Pi_a^0(T)=\Pi_a(T)=\Pi^0(T)=\Pi(T).\] \end{lem}

\begin{proof} Suppose that $T$ satisfies property $(Z_{\Pi_{a}}),$
that's $\sigma(T)=\sigma_{W}(T)\sqcup\Pi_a(T).$ Then
$\mu\in\Pi_a(T)\Longleftrightarrow\mu\in\mbox{iso}\,\sigma_a(T)\cap\sigma_{W}(T)^C$
$\Longrightarrow \mu\in\Pi_a^0(T).$  This implies that
$\Pi_a(T)=\Pi_a^0(T)$  and  $\Pi(T)=\Pi^0(T).$ So  $\sigma(T)=\sigma_{W}(T)\sqcup\Pi_a^0(T)$ and this implies  that $\Pi^0(T)=\Pi_a^0(T).$ Therefore
$\Pi(T)=\Pi^0(T)=\Pi_a(T)=\Pi_a^0(T).$\end{proof}

In the following theorem we establish a relationship between the property $(Z_{\Pi_{a}}),$ the properties $(gab),$ $(ab),$  $(Bab)$
and the classical Browder's theorem (see Definition\ref{dfn0}).

\begin{thm}\label{thm2}Let $T\in \mathcal{B}(X)$. Then the following statements are equivalent:\\
i) $T$ satisfies property
$(Z_{\Pi_{a}});$\\
ii) $T$ satisfies property $(gab)$ and $\sigma_{BW}(T)=\sigma_{W}(T);$\\
iii) $T$ satisfies property $(ab)$ and $\Pi_a(T)=\Pi_a^0(T);$\\
iv) $T$ satisfies property $(Bab)$ and $\Pi_a(T)=\Pi_a^0(T).$\\
v) $T$ satisfies Browder's theorem and $\Pi_a(T)=\Pi^0(T).$\end{thm}

\begin{proof} (i) $\Longleftrightarrow$ (ii)  Suppose that $T$ satisfies property $(Z_{\Pi_{a}}),$
that's $\sigma(T)=\sigma_{W}(T)\sqcup\Pi_a(T).$ From Lemma \ref{lem2}, $\sigma(T)=\sigma_{W}(T)\sqcup\Pi_a^0(T).$
So  $T$ satisfies
property $(ab)$.  As
$\Pi(T)=\Pi_a(T)$, then from \cite[Theorem 2.8]{BZ3}, $T$ satisfies
property $(gab)$. Moreover,
$\sigma_{BW}(T)=\sigma(T)\setminus\Pi_a(T)=\sigma_{W}(T)$. The
reverse implication is obvious.\\
(i) $\Longleftrightarrow$ (iii) Follows directly from Lemma \ref{lem2}.\\
 (i) $\Longleftrightarrow$ (iv) If $T$ satisfies property $(Z_{\Pi_{a}}),$ then
  $\sigma(T)\setminus\sigma_{BW}(T)=\sigma(T)\setminus\sigma_{W}(T)=\Pi_a^0(T)=\Pi_a(T).$ So $T$ satisfies property $(Bab).$
  Conversely, the property $(Bab)$ for $T$ implies   from   \cite[Theorem 3.6]{ZZ} that  $\sigma_{BW}(T)=\sigma_{W}(T).$
   So $\sigma_{W}(T)=\sigma(T)\setminus\Pi_a^0(T)=\sigma(T)\setminus\Pi_a(T)$ and this means that $T$ satisfies property $(Z_{\Pi_{a}}).$
  The equivalence between assertions (i) and  (v) is clear.\end{proof}

\begin{rema}
 From   Theorem \ref{thm2},  It follows that:
 \begin{enumerate}
 \item If $T$ satisfies property $(Z_{\Pi_{a}}),$ then it satisfies property $(gab)$
 and then property $(ab)$ and Browder's theorem. But the converses are not true in general.  For this,
let $T\in L(\ell^2(\N))$ be defined by $T(x_1, x_2, x_3, ...)=(0, 0,
\frac{1}{3}x_1, 0, 0, ...).$ Thus $\sigma(T)=\sigma_{W}(T)=\{0\}$
and $\Pi_a(T)=\{0\}$ and since $T$ is nilpotent, then
$\sigma_{BW}(T)=\emptyset.$ So  $T$
satisfies property $(gab)$ and then property $(ab)$ and Browder's theorem. But $T$ does not satisfy property
$(Z_{\Pi_{a}}).$
\item Also  we  cannot expect that property $(Z_{\Pi_{a}})$ holds for an operator satisfying property
$(Bab),$ as we can see in the following: It is easily seen that the operator $T$ defined in the second point of Remark \ref{rema1}
satisfies property $(Bab)$ and it does not satisfy property $(Z_{\Pi_{a}}).$ Here $\Pi_a(T)=\{0\}$ and $\Pi_a^0(T)=\emptyset.$
\end{enumerate}
\end{rema}

The following property has relevant role in local
spectral theory: a bounded linear
 operator $T\in \mathcal{B}(X)$ is said to have the {\it single-valued
 extension property} (SVEP for short) at $\lambda\in\mathbb{C}$ if
  for every open neighborhood $U_\lambda$ of $\lambda,$ the  function $f\equiv 0$ is the only
 analytic solution of the equation
 $(T-\mu I)f(\mu)=0\quad\forall\mu\in U_\lambda.$ We denote by
  ${\mathcal S}(T)=\{\lambda\in\mathbb{C}: T\mbox{  does not have SVEP at } \lambda\}$
     and we say that  $T$ has  SVEP
  if
$ {\mathcal S}(T)=\emptyset.$ We say that $T$ has SVEP on $A\subset\mathbb{C},$ if $T$ has SVEP
at every $\lambda\in A.$ ( For more details about this property, we refer the reader to \cite{LN}).

\begin{prop}\label{thm0} Let $T\in \mathcal{B}(X).$ If   $T$ or its dual  $T^*$ has SVEP on
 ${\sigma_{W}(T)}^C$ then  $T$ satisfies
property  $(Z_{E_{a}})$ if and only if  $E_a(T)=\Pi^0(T);$ where ${\sigma_{W}(T)}^C$
 is the complement of the Weyl spectrum of $T.$\end{prop}

 \begin{proof} If $T$ satisfies property  $(Z_{E_{a}}),$ then from Lemma \ref{lem0},  $E_a(T)=\Pi^0(T).$ Remark
that in this implication, the condition of SVEP for $T$ or $T^*$ is not necessary. Conversely, assume that $E_a(T)=\Pi^0(T).$
 Note that   $T$ has SVEP  on
 ${\sigma_{W}(T)}^C$ $\Longleftrightarrow$ $T^*$ has SVEP  on
 ${\sigma_{W}(T)}^C.$ But this is equivalent to say that   $T$ satisfies Browder's theorem $\sigma(T)\setminus\sigma_{W}(T)=\Pi^0(T)=E_a(T).$
 So $T$ satisfies property  $(Z_{E_{a}}).$ \end{proof}

 \begin{rema}The hypothesis   $T$ or $T^*$ has SVEP on  ${\sigma_{W}(T)}^C$  is crucial as  shown in this example:
define the operator  $T$ by    $T=R\oplus R^*$.  We have  $\sigma(T)
= D(0, 1)$ and $\Pi^0(T)=\emptyset.$ But, since  $ \dim \NN(T)=\mbox{codim}\mathcal{R}(T)=1,$ then $0\not\in\sigma_{W}(T).$ So $T$ does not satisfy property $(Z_{E_{a}}).$
Note that $T$ and $T^*$ do not have SVEP at
 $0\in {\sigma_{W}(T)}^C,$ as
 ${\mathcal S}(T)={\mathcal S}(T^*)={\mathcal S}(U^*) =\{\lambda \in\mathbb{C}:
0\leq|\lambda|<1\}.$
 \end{rema}

 Similarly, we have the following proposition for the property  $(Z_{\Pi_{a}}).$

 \begin{prop} If   $T\in \mathcal{B}(X)$ or its dual $T^*$ has SVEP on
 ${\sigma_{W}(T)}^C$ then  $T$ satisfies
property  $(Z_{\Pi_{a}})$ if and only if  $\
\Pi_a(T)=\Pi^0(T).$\end{prop}

 \begin{proof} Obtained by an argument similar to the one of the  preceding proof.
 \end{proof}

Now,  we give a summary of the results obtained in this section. In the following diagram which is a combination  with the first presented above,
 arrows signify implications and
the numbers near the arrows are references to the results obtained in  in
this section (numbers without brackets) or to the bibliography therein (the
numbers in square brackets).

\vspace{5pt}

\vbox{
\[
\begin{CD}@.@.@.(Baw)@.(Z_{E_a})@>\mbox{{\scriptsize\ref{cor2}}}>>(k)\\
 @.@.@.@VV\mbox{{\scriptsize\cite{ZZ}}}V@.@VV\mbox{{\scriptsize\cite{BE-Kachad}}}V\\
 (Baw)@<\mbox{{\scriptsize\ref{thm1}}}<<(Z_{E_a})@>\mbox{{\scriptsize\ref{thm1}}}>>(gaw)
@>\mbox{{\scriptsize\cite{BZ3}}}>> (aw) @>\mbox{{\scriptsize\cite
{BZ1}}}>>\mbox{$W$}@< \mbox{{\scriptsize\cite
{BK}}}<<\mbox{$gW$} \\
@VV\mbox{{\scriptsize\cite{ZZ}}}V @VV\mbox{{\scriptsize\ref{cor2}}}V@VV\mbox{{\scriptsize\cite{ZZ}}}V@VV\mbox{{\scriptsize\cite{BZ3}}}V
@VV\mbox{{\scriptsize\cite{Barnes}}}V@VV\mbox{{\scriptsize\cite{BE1}}}V\\
(Bab)@<\mbox{{\scriptsize\ref{thm2}}}<<(Z_{\Pi_a})@>>\mbox{{\scriptsize\ref{thm2}}}> (gab)
@>>\mbox{{\scriptsize\cite{BZ3}}}> (ab)
@>>\mbox{{\scriptsize\cite{BZ3}}}> \mbox{$B$}& \Longleftrightarrow_{\mbox{{\scriptsize\cite{AmZg1}}}} &\mbox{  $gB$}\\
@.@.@.@AA\mbox{{\scriptsize\cite{ZZ}}}A\\
@.@.@.(Bab)\\
\end{CD}
\]
}

\section{Preservation under commuting  Riesz perturbations}
We recall that an operator $R\in \mathcal{B}(X)$ is
said to
  be \textit{Riesz} if $R-\mu I$ is Fredholm for every
  non-zero complex $\mu,$ that is, $\pi(R)$ is quasinilpotent in the
  Calkin algebra $C(X)=\mathcal{B}(X)/\kx$ where  $\pi$ is the canonical mapping of
  $ \mathcal{B}(X)$ into $C(X).$

 We denote by  $\f^0x,$  the class of finite rank power operators as follows:
   \[\f^0x=\{ S\in \mathcal{B}(X) : S^n \in \fx \mbox { for some } n\in \mathbb{\mathbb{N}} \}.\]
    Clearly, \[\fx\cup \nx\subset \f^0x\subset \rx, \mbox{ and }  \kx\cup \qx\subset \rx.\]

We start this section by the following nilpotent  perturbation result.

\begin{prop}\label{prop1} Let $T\in \mathcal{B}(X)$ and let $N\in \mathcal{N}(X)$ which commutes with $T.$
Then $T$ satisfies property $(s)$ if and only if $T+N$ satisfies property $(s);$  where $s\in\{ Z_{E_{a}}, Z_{\Pi_{a}}\}.$ \end{prop}

\begin{proof}Since $N$ is nilpotent and commutes with $T,$ we know that  $\sigma(T+N)=\sigma(T).$  From the proof of \cite[Theorem 3.5]{BZ1},
we have  $0<n(T+N)\Longleftrightarrow 0<n(T)$ and so
  $E_a(T+N)=E_a(T).$
    From \cite[Corollary 3.8]{ZJZ} we
   know that  $\Pi_a(T+N)=\Pi_a(T).$   Furthermore,  $\sigma_{W}(T+N)=\sigma_{W}(T),$ see
  \cite[Lemma 2.2]{Ob}. This finishes the proof.
\end{proof}

\begin{rema}\rm   We   notice that the assumption of commutativity  in the  Proposition \ref{prop1}   is crucial.
\begin{enumerate}
\item  Let $T$ and $N$ be defined on $\ell^2(\mathbb{N})$ by
$$ T(x_1, x_2, \ldots)=(0, \frac {x_1}{2}, \frac{x_2}{3}, \ldots) \mbox{ and }
 N(x_1, x_2, \ldots)=(0, \frac{-x_1}{2}, 0, 0, \ldots).$$
 Clearly $N$ is nilpotent  and does not commute with $T.$ The
property $(Z_{E_{a}})$ is  satisfied by $T,$ since
$\sigma(T)=\{0\}=\sigma_{W}(T)$ and $E_a(T)=\emptyset.$ But $T+N$
does not satisfy  property $(Z_{E_{a}})$  as we have
$\sigma(T+N)=\sigma_{W}(T+N)=\{0\}$ and $\{0\}=  E_a(T+N).$
\item Let $T$ and $N$ be defined by $$T(x_1,
x_2, x_3, \ldots)=(0, x_1, x_2, x_3, \ldots) \mbox{ and }
N(x_1, x_2, \ldots)=(0,-x_1, 0, 0, \ldots).$$ $N$ is nilpotent and $TN\neq NT.$ Moreover,
$\sigma(T)=\sigma_{W}(T)=D(0, 1),$ and  $\Pi_a(T)=\emptyset.$ So $T$ satisfies property $(Z_{\Pi_{a}}).$
 But $T+N$ does not satisfy property $(Z_{\Pi_{a}}),$ since   $\sigma(T+N)=\sigma_{W}(T+N)=D(0, 1),$ and $\Pi_a(T+N)=\{0\}.$
\end{enumerate} \end{rema}

The stability of properties $(Z_{E_{a}})$ and $(Z_{\Pi_{a}}),$ showed in  Proposition \ref{prop1} cannot
 be extended   to commuting quasi-nilpotent operators, as we can see in the next Example.

 \begin{ex}\label{} \rm   Let $R$ be the   operator  defined on $\ell^2(\mathbb{N})$ by
$R(x_1,x_2, \ldots)=(0, \frac{x_1}{2},\frac{x_2}{3}, \ldots)$ and let  $T$ be the operator defined on
$\ell^2(\mathbb{N})$
 by $T=-R.$
 Clearly $R$ is compact and quasi-nilpotent  and verifies   $TR=RT=-R^2.$ Moreover, $T$ satisfies properties $(Z_{E_{a}})$ and $(Z_{\Pi_{a}}),$ because $\sigma(T)=\{0\}=\sigma_{W}(T)$
 and $E_a(T)=\emptyset.$ But $T+R=0$ does not satisfy neither  property $(Z_{E_{a}})$ nor property $(Z_{\Pi_{a}}),$ since
 $\sigma(T+R)=\{0\}=\sigma_{W}(T+R)$ and $E_a(T+R)=\{0\},$  $\Pi_a(T+R)=\{0\}.$ Here $\Pi^0(T+R)=\emptyset.$\end{ex}

 However, in the next theorems,
 we give necessary and sufficient conditions to ensure the stability of these properties  under commuting
 perturbations by Riesz operators which are not necessary nilpotent. The case of nilpotent operators is studied in
  Proposition \ref{prop1}.

  \begin{thm} \label{thm4}Let $R\in \rx$ and let $T\in \mathcal{B}(X)$ which commutes with $R.$
 If $T$ satisfies property $(Z_{E_{a}}),$ then the following statements are equivalent:\\
i)  $T+R$ satisfies property $(Z_{E_{a}});$\\
 ii) $E_a(T+R)=\Pi^0(T+R);$\\
 iii) $E_a(T+R)\cap \sigma(T)\subset \Pi^0(T).$
 \end{thm}

\begin{proof} i) $\Longleftrightarrow$ ii) If $T+R$  satisfies $(Z_{E_{a}}),$ then from Lemma \ref{lem0}
we have $E_a(T+R)=\Pi^0(T+R).$ Conversely, assume that
$E_a(T+R)=\Pi^0(T+R).$ Since $T$ satisfies property  $(Z_{E_{a}})$ then
 it satisfies  Browder's theorem. From \cite[Lemma 3.5]{AZ},  $T+R$ satisfies Browder's theorem
  that's $\sigma(T+R)\setminus\sigma_{W}(T+R)=\Pi^0(T+R).$ So
  $T+R$ satisfies property $(Z_{E_{a}}).$\\
 ii) $\Longrightarrow$ iii) Assume that $\Pi^0(T+R)=E_a(T+R)$ and
let $\lambda_{0}\in E_a(T+R)\cap \sigma(T)$ be arbitrary. Then
$\lambda_{0}\in \Pi^0(T+R)\cap\sigma(T)$ and so
$\lambda_{0}\not\in\sigma_{b}(T+R).$ Since w know from \cite{Ra1} that $\sigma_{b}(T)=\sigma_{b}(T+R),$  then
$\lambda_{0}\in\Pi^0(T).$ This proves that $E_a(T+R)\cap
\sigma(T)\subset \Pi^0(T).$

iii) $\Longrightarrow$ ii) Suppose that $E_a(T+R)\cap\,
\sigma(T)\subset \Pi^0(T).$ As the inclusion $E_a(T+R)\supset \Pi^0(T+R)$ is always true, it suffices to show that $E_a(T+R)\subset
\Pi^0(T+R).$ Let $\mu_{0}\in E_a(T+R)$ be arbitrary. We distinguish
two cases: the first is $\mu_{0}\in\sigma(T).$ Then $\mu_{0}\in E_a(T+R)\cap\,\sigma(T)\subset \Pi^0(T).$ So
$\mu_{0}\not\in\sigma_{b}(T)=\sigma_{b}(T+R)$ and  then $\mu_{0}\in\Pi^0(T+R).$ The second case
is $\mu_{0}\not\in\sigma(T).$ This implies that
$\mu_{0}\not\in\sigma_{b}(T+R).$ Thus
 $\mu_{0}\in\Pi^0(T+R).$
As a conclusion,  $E_a(T+R)= \Pi^0(T+R).$
Remark that the statements ii) and iii) are  always equivalent without the assumption that $T$ satisfies
property  $(Z_{E_{a}}).$   \end{proof}





Similarly to Theorem \ref{thm4}, we have the following perturbation result for
 the property $(Z_{\Pi_{a}}).$

\begin{thm}\label{thm6} Let $R\in \rx.$
 If $T\in \mathcal{B}(X)$ satisfies property $(Z_{\Pi_{a}})$  and commutes with $R,$ then
the following statements are equivalent:\\
i) $T+R$ satisfies property $(Z_{\Pi_{a}});$ \\
ii) $\Pi^0(T+R)=\Pi_a(T+R);$\\
iii) $\Pi_a(T+R)\cap \sigma(T)\subset \Pi^0(T).$\end{thm}

\begin{proof} i) $\Longleftrightarrow$ ii) If $T+R$ satisfies $(Z_{\Pi_{a}})$ then from Lemma \ref{lem2},
$\Pi_a(T+R)=\Pi^0(T+R).$ Conversely, suppose that
$\Pi_a(T+R)=\Pi^0(T+R).$ Since $T$ satisfies property  $(Z_{\Pi_{a}})$ then
 it satisfies  Browder's theorem. Hence $T+R$ satisfies Browder's theorem that's $\sigma(T+R)\setminus\sigma_{W}(T+R)=\Pi^0(T+R).$ So
  $T+R$ satisfies property $(Z_{\Pi_{a}}).$\\
ii) $\Longleftrightarrow$ iii) Goes similarly with the proof of the equivalence between the second and the third statements of  Theorem \ref{thm4}.
Notice also that this equivalence is always true without property  $(Z_{\Pi_{a}})$ for $T.$ \end{proof}

The following   example  proves  in general that, the properties $(Z_{E_{a}})$ and $(Z_{\Pi_{a}})$  are not preserved under
commuting   finite rank power perturbations.

\begin{ex} \label{ex2} \rm
For fixed $0<\varepsilon<1,$ let $F_{\varepsilon}$ be a finite rank operator defined on $\ell^2(\mathbb{N})$ by
$F_{\epsilon}(x_1, x_2, x_3, \ldots)=(-\varepsilon x_1, 0, 0, 0, \ldots).$  We consider  the operators $T$ and $F$ defined by $T = R\oplus I$ and
$F = 0 \oplus F_{\varepsilon}.$
 $F$ is a finite rank operator and $TF = FT.$ We have,
$$\sigma(T)=\sigma(R)\cup\sigma(I)=D(0, 1) \mbox{, }\sigma_a(T)=\sigma_a(R)\cup\sigma_a(I)=C(0, 1) \mbox{, }
\sigma_{W}(T)=D(0, 1),$$
$$\sigma(T+F)=\sigma(R)\cup\sigma(I+F_{\varepsilon})=D(0, 1) \mbox{, } \sigma_{W}(T+F)=D(0, 1)\mbox{ and } $$  $$\sigma_a(T+F)=
\sigma_a(R)\cup\sigma_a(I+F_{\varepsilon})=C(0, 1)\cup\{1-\varepsilon\}.$$ Moreover,
   $E_a(T)=\Pi_a(T)=\emptyset.$ So  $T$ satisfies properties $(Z_{E_{a}})$ and $(Z_{\Pi_{a}}).$
 But, since $E_a(T+F)=\Pi_a(T+F)=\{1-\varepsilon\},$ then  $T+F$ does not satisfy either property
 $(Z_{E_{a}})$ nor property $(Z_{\Pi_{a}}).$ Here $\Pi_a(T+F)\cap\sigma(T)=\{1-\varepsilon\},$ $\Pi^0(T)=\Pi^0(T+F)=\emptyset.$  \end{ex}

\goodbreak
{\small \noindent Hassan  Zariouh,\newline Centre R\'egional pour les M\'etiers
de l'\'Education\newline et de la Formation
 de la r\'egion de l'oriental (CRMEFO),\newline Bodir, Oujda, Maroc.\newline
 \noindent h.zariouh@yahoo.fr
\end{document}